\documentclass[12pt,leqno]{amsart}
\usepackage{amsmath,amsfonts,amssymb,amsthm}
\usepackage{graphicx}
\graphicspath{ }
\usepackage{wrapfig}
\usepackage{accents}
\usepackage{caption}
\usepackage{subcaption}
\usepackage{calligra}

\numberwithin{figure}{section}

\theoremstyle{plain}
\newtheorem{thm}{Theorem}[section]
\newtheorem{lem}[thm]{Lemma}

\newtheorem{cor}{Corollary}[thm]
\theoremstyle{definition}

\theoremstyle{remark}

\usepackage{mathtools}
\title[On Ricci solitons whose potential is convex]{On Ricci solitons whose potential is convex}
\author[C. K. Mondal, A. A. Shaikh]{Chandan Kumar Mondal$^1$, Absos Ali Shaikh$^2$}
\address{\noindent\newline $^{12}$Department of Mathematics,\newline University of Burdwan, Golapbag,\newline Burdwan-713104,\newline West Bengal, India}
\email{$^1$chan.alge@gmail.com}
\email{$^2$aask2003@yahoo.co.in, aashaikh@math.buruniv.ac.in}
\begin{document}
\begin{abstract}
In this paper we consider the Ricci curvature of a Ricci soliton. In particular, we have showed that a complete gradient Ricci soliton with non-negative Ricci curvature possessing a non-constant convex potential function having finite weighted Dirichlet integral satisfying an integral condition is Ricci flat and also it isometrically splits a line. We have also proved that a gradient Ricci soliton with non-constant concave potential function and bounded Ricci curvature is non-shrinking and hence the scalar curvature has at most one critical point.
\end{abstract}
\noindent\footnotetext{$\mathbf{2010}$\hspace{5pt}Mathematics\; Subject\; Classification: 53C20; 53C21; 53C44.\\ 
{Key words and phrases: Ricci soliton; scalar curvature; Ricci flat; convex function; critical point; Riemannian manifold. } }
\maketitle
\section{Introduction and preliminaries}
In 1982, Hamilton \cite{HA82} introduced the concept of Ricci flow. The Ricci
flow is defined by an evolution equation for metrics on the Riemannian manifold $(M,g_0)$:
$$\frac{\partial}{\partial_t}g(t)=-2Ric,\quad g(0)=g_0.$$
 A complete Riemannian manifold $(M,g)$ of dimension $n\geq 2$ with Riemannian metric $g$ is called a Ricci soliton if there exists a vector field $X$ satisfying
\begin{equation}\label{r7}
Ric+\frac{1}{2}\pounds_Xg=\lambda g,
\end{equation}
where $\lambda$ is a constant and $\pounds$ denotes the Lie derivative. The vector field $X$ is called potential vector field. The Ricci solitons are self-similar solutions to the Ricci flow. Ricci solitons are natural generalization of Einstein metrics, which have been significantly studied in differential geometry and geometric analysis. A Ricci soliton is an Einstein metric if the vector field $X$ is zero or Killing. Throughout the paper by $M$ we mean an $n$-dimensional, $n\geq 2$, complete Riemannian manifold endowed with Riemannian metric $g$. Let $C^\infty(M)$ be the ring of smooth functions on $M$. If $X$ is the gradient of some function $u\in C^\infty(M)$, such a manifold is called a \textit{gradient Ricci soliton}, and then (\ref{r7}) reduces to the form
\begin{equation}\label{r1}
\nabla^2u+Ric=\lambda g,
\end{equation}
where $\nabla^2u$ is the Hessian of $u$ and the function $u$ is called potential function. The Ricci soliton $(M,g,X,\lambda)$ is called shrinking, steady and expanding according as $\lambda>0$, $\lambda=0$ and $\lambda<0$, respectively. Each type of Ricci solitons determines some unique topology of the manifold. For example, if the scalar curvature of a complete gradient shrinking Ricci soliton is bounded, then the manifold has finite topological type \cite{FMZ08}. Munteanu and Wang proved that an $n$-dimensional gradient shrinking Ricci soliton with non-negative sectional curvature and positive Ricci curvature must be compact \cite{MW2017}, for more results see \cite{MW17,MW017}. Perelman \cite{PE03} proved that a compact Ricci soliton is always gradient Ricci soliton. For the detailed treatment on Ricci solitons and their interaction to Ricci flow, we refer to \cite{CZ06, CK04}. A smooth function $\varphi:M\rightarrow\mathbb{R}$ is said to be convex \cite{CK19,YA74} if for any $p\in M$ and for any vector $v\in T_pM$
$$\left\langle grad \varphi,v\right\rangle_p\leq \varphi(exp_pv)-\varphi(p).$$ 
If $\varphi$ is convex, then $-\varphi$ is called concave.	
\par The paper is arranged as follows: In the first section, we have proved that a complete non-compact gradient Ricci soliton with non-negative Ricci curvature possessing a non-constant convex function with finite weighted Dirichlet integral satisfying an integral condition is Ricci flat and also it isometrically splits a line. We have also deduced a corollary relating to the Ricci soliton and harmonic function. In the last section, we have proved that if in a complete gradient Ricci soliton, the potential function is a non-constant concave function with bounded Ricci curvature then the scalar curvature possesses at most one critical point, see Theorem \ref{th3}.  
\section{Ricci soliton and Ricci flat manifold}

\begin{lem}\label{th1}
Let $(M,g)$ be a complete Riemannian manifold with non-negative Ricci curvature. If $u\in C^\infty(M)$ is a non-constant convex function with finite weighted Dirichlet integral, i.e., 
\begin{equation}\label{r6}
\int_{M-B(p,r)}d(x,p)^{-2}|\nabla u|^2<\infty,
\end{equation}
 and also satisfies the relation
 \begin{equation}\label{eq1}
 \int_{M-B(p,r)}d(x,p)^{-2}u<\infty,
 \end{equation}
where $B(p,r)$ is an open ball with center $p$ and radius $r$, then the hessian of $u$ vanishes in $M$.
\end{lem}
\begin{proof}
Since $u\in C^\infty(M)$ is a non-constant convex function on $M$, it follows that \cite{YA74}, $M$ is non-compact. Now, we consider the cut-off function, introduced in \cite{CC96}, $\varphi_r\in C^2_0(B(p,2r))$ for $r>0$ such that
\[ \begin{cases} 
	  0\leq \varphi_r\leq 1 &\text{ in }B(p,2r)\\
      \varphi_r=1  & \text{ in }B(p,r) \\
      |\nabla \varphi_r|^2\leq\frac{C}{r^2}& \text{ in }B(p,2r) \\
      \Delta \varphi_r\leq \frac{C}{r^2} &  \text{ in }B(p,2r).
   \end{cases}
\]
Then for $r\rightarrow\infty$, we have $\Delta \varphi^2_r\rightarrow 0$ as $\Delta \varphi^2_r\leq \frac{C}{r^2}$.
Since $u$ is a smooth convex function, $u$ is also subharmonic \cite{GH71}, i.e., $\Delta u\geq 0$. Now using integration by parts, we have
\begin{equation}\label{r5}
\int_M u\Delta \varphi^2_r=\int_M \Delta u\varphi^2_r.
\end{equation}
 Since $\varphi_r\equiv 1$ in $B(p,r)$, using (\ref{r5}), we get
 \begin{equation*}
 \int_{B(p,r)}\Delta u=0.
 \end{equation*}
 Again, using the integration by parts and also by our assumption, we obtain
 \begin{equation*}
0\leq  \int_{B(p,2r)}\varphi_r^2\Delta u=\int_{B(p,2r)-B(p,r)}u\Delta \varphi_r^2\leq \int_{B(p,2r)-B(p,r)}u\frac{C}{r^2}\rightarrow 0,
 \end{equation*}
 as $r\rightarrow \infty$. Hence we have
 $$\int_M\Delta u=0.$$ But $\Delta u\geq 0$. Therefore, $\Delta u=0$ in $M$, i.e., $u$ is a harmonic function.
The Bochner formula \cite{AU13} for the Riemannian manifold is written as
$$\frac{1}{2}\Delta|\nabla u|^2=|\nabla^2 u|^2+g(\nabla u,\nabla\Delta u)+Ric(\nabla u,\nabla u).$$
Since $u$ is harmonic, so $\Delta u=0$. Therefore, the above equation reduces to
\begin{equation}\label{r1.1}
\frac{1}{2}\Delta|\nabla u|^2=|\nabla^2 u|^2+Ric(\nabla u,\nabla u).
\end{equation}
 Combining $\varphi_r^2 $ with (\ref{r1.1}) and then integrating we obtain
$$\int_M\Big\{|\nabla^2 u|^2+Ric(\nabla u,\nabla u)\Big\}\varphi^2_r=\int_M\frac{1}{2}\Delta|\nabla u|^2\varphi^2_r.$$
Using integration by parts we get
$$\int_M\frac{1}{2}\Delta|\nabla u|^2\varphi^2_r=\int_M \frac{1}{2}|\nabla u|^2\Delta \varphi^2_r.$$
Then the above equation and the property of $\varphi_r$ together imply
$$\int_{B(p,2r)-B(p,r)}\frac{1}{2}\Delta|\nabla u|^2\varphi^2_r\leq \int_{B(p,2r)-B(p,r)}\frac{C}{2r^2}|\nabla u|^2\rightarrow 0$$
as $r\rightarrow \infty$. And also in $B(p,r)$ we have
$$\int_{B(p,r)}\Big\{|\nabla^2 u|^2+Ric(\nabla u,\nabla u)\Big\}=\int_{B(p,r)}\frac{1}{2}|\nabla u|^2\Delta \varphi^2_r=0,$$
since $\varphi^2_r \equiv 1 $ in $B(p,r)$. Therefore
\begin{eqnarray}
\nonumber\int_M\Big\{|\nabla^2 u|^2+Ric(\nabla u,\nabla u)\Big\}&=& \lim_{r\rightarrow\infty}\Big(\int_{B(p,2r)-B(p,r)}+\int_{B(p,r)}\Big)\Big\{|\nabla^2 u|^2+Ric(\nabla u,\nabla u)\Big\}\\
&=& 0,
\end{eqnarray} 
which implies that $\nabla^2 u=0$.
\end{proof}
\begin{lem}\cite[Lemma 2.3]{SA96}\label{lm1}
Let $u$ be a smooth function in a complete Riemannian manifold $(M,g)$. Then the following conditions are equivalent:\\
$(i)$ $u$ is an affine function,\\
$(ii)$ Hessian of $u$ vanishes everywhere in $M$,\\
$(iii)$ $\nabla u$ is a Killing vector field with $|\nabla u|$ is constant. 
\end{lem}
\begin{thm}\cite[Theorem 1]{IN82}\label{th5}
If a complete Riemannian manifold $(M,g)$ admits a non-constant smooth affine function, then $M$ is isometric to $N\times\mathbb{R}$ for a totally geodesic submanifold $N$ of $M$.
\end{thm}
\begin{thm}\label{th4}
Let $M$ be a complete Riemannian manifold with non-negative Ricci curvature. If $M$ admits a non-constant convex function satisfying (\ref{r6}) and (\ref{eq1}), then $M$ is isometric to the Riemannian product $N\times\mathbb{R}$, where $N$ is a totally geodesic submanifold of $M$.
\end{thm}
\begin{proof}
In view of Lemma \ref{th1}, it follows that hessian of $u$ vanishes. Again, Lemma \ref{lm1} implies that $u$ is an affine function. Therefore, using Theorem \ref{th5}, we conclude that $M$ is isometric to the Riemannian product $N\times\mathbb{R}$, where $N$ is a totally geodesic submanifold of $M$.
\end{proof}
\begin{thm}\label{th6}
Let $(M,g,u)$ be a complete gradient Ricci soliton with non-negative Ricci curvature. If $u$ is a non-constant convex function on $M$ satisfying (\ref{r6}) and (\ref{eq1}), then $M$ is Ricci flat. Moreover, $\nabla u$ is a Killing vector field with $|\nabla u|$ is constant.
\end{thm}
\begin{proof}
 From Lemma \ref{th1}, we get $\nabla^2 u=0$ and $Ric(\nabla u,\nabla u)=0$, since Ricci curvature is non-negative. Now from (\ref{r1}), we get
 $$Ric(\omega,\omega)=\lambda g(\omega,\omega),\text{ for }\omega\in TM,$$
 which implies that
$$Ric(\nabla u,\nabla u)=\lambda g(\nabla u,\nabla u)=0.$$
Therefore, we conclude that $\lambda=0$. And hence, (\ref{r1}) implies that Ricci curvature of $M$ vanishes in $M$, i.e., $M$ is a Ricci flat manifold. Also, Lemma \ref{th1} and Lemma \ref{lm1} together imply that $\nabla u$ is Killing vector field with $|\nabla u|$ is constant.
\end{proof}

\begin{cor}\label{co1}
Let $(M,g,u)$ be a complete non-compact gradient Ricci soliton  satisfying (\ref{r1}) with non-negative Ricci curvature. If $u\in C^\infty(M)$ is a harmonic function with finite weighted Dirichlet integral, i.e., 
$$\int_{M-B(p,r)}d(x,p)^{-2}|\nabla u|^2<\infty,$$
then $M$ is a Ricci flat manifold.
\end{cor}
\begin{proof}
The proof is same as that of Theorem \ref{th6} except the part where we have proved the harmonicity of the function $u$ and hence we omit.
\end{proof}
\section{Ricci soliton and critical points}
\begin{thm}\label{th2}
Let $(M,g)$ be a complete gradient Ricci soliton satisfying (\ref{r1}). If $u\in C^\infty(M)$ is a non-constant concave function and $(M,g)$ has bounded Ricci curvature, i.e., $|Ric|\leq K$ for some constant $K> 0$, then the Ricci soliton is non-shrinking. 
\end{thm}
\begin{proof}
Since $u$ is a non-trivial concave function in $M$, the function $-u$ is non-constant convex and it implies that the manifold $M$ is non-compact. Let us consider a length minimizing normal geodesic $\gamma:[0,t_0]\rightarrow M$ for some arbitrary large $t_0>0$. Take $p=\gamma(0)$ and $X(t)=\gamma'(t)$ for $t>0$. Then $X$ is the unit tangent vector along $\gamma$. Now integrating (\ref{r1}) along $\gamma$, we get
\begin{eqnarray}
\int_{0}^{t_0}Ric(X,X)&=& \int_{0}^{t_0}\lambda g(X,X)-\int_{0}^{t_0}\nabla^2 u(X,X)\nonumber \\
&=& \lambda t_0-\int_{0}^{t_0}\nabla^2 u(X,X).
\end{eqnarray}
Again, by the second variation of arc length, we have
\begin{equation}
\int_{0}^{t_0}\varphi^2Ric(X,X)\leq (n-1)\int_{0}^{t_0}|\varphi'(t)|^2dt,
\end{equation}
for every non-negative function $\varphi$ defined on $[0,t_0]$ with $\varphi(0)=\varphi(t_0)=0$. We now choose the function $\varphi$ as the following:
\[\varphi(t)= \begin{cases} 
      t & t\in[0,1] \\
      1 & t\in[1,t_0-1] \\
      t_0-t & t\in[t_0-1,t_0].
   \end{cases}
\]
Then
\begin{eqnarray}\label{r2}
\int_{0}^{t_0}Ric(X,X)dt&=& \int_{0}^{t_0}\varphi^2 Ric(X,X)dt+\int_{0}^{t_0}(1-\varphi^2)Ric(X,X)dt\nonumber \\
&\leq & (n-1)\int_{0}^{t_0}|\varphi'(t)|^2dt+\int_{0}^{t_0}(1-\varphi^2)Ric(X,X)dt\nonumber \\
&\leq & 2(n-1)+\sup_{B(p,1)}|Ric|+\sup_{B(\gamma(t_0),1)}|Ric|.
\end{eqnarray}
Combining the equations (\ref{r1}) and (\ref{r2}), we get
\begin{eqnarray}\label{r3}
\lambda t_0-\int_{0}^{t_0}\nabla^2 u(X,X)&\leq & 2(n-1)+\sup_{B(p,1)}|Ric|+\sup_{B(\gamma(t_0),1)}|Ric|\nonumber \\
&=& 2(n-1)+2K.
\end{eqnarray}
Therefore, taking limit $t_0\rightarrow \infty$ on both sides of (\ref{r3}), we can write
\begin{equation}\label{r4}
\lim\limits_{t_0\rightarrow\infty}\lambda t_0-\lim\limits_{t_0\rightarrow\infty}\int_{0}^{t_0}\nabla^2 u(X,X)\leq 2(n-1)+2K.
\end{equation}
Now $\lim\limits_{t_0\rightarrow\infty}\int_{0}^{t_0}\nabla^2 u(X,X)\leq 0$, since $u$ is a concave function. If $\lambda>0$, then $\lim\limits_{t_0\rightarrow\infty}\lambda t_0=+\infty$, which contradicts the inequality (\ref{r4}). Thus $\lambda\leq 0$, i.e., the Ricci soliton is non-shrinking.
\end{proof}
\begin{lem}\label{le1}\cite{GH10}
Let $(M,g)$ be a steady gradient Ricci soliton with positive Ricci curvature. Then there is atmost one critical point of $R$.
\end{lem}
\begin{thm}\label{th3}
Let $(M,g)$ be a complete non-compact gradient Ricci soliton satisfying 
$$\nabla^2u+Ric=\lambda g,$$
with $\lambda \geq 0$. If $u\in C^\infty(M)$ is a non-constant concave function and Ricci curvature of $M$ satisfies $0<Ric\leq K$ for some constant $K> 0$, then there is atmost one critical point of the scalar curvature $R$.
\end{thm}
\begin{proof}
Since $\lambda\geq 0$, by using Theorem \ref{th2} we can prove that $\lambda=0$. Therefore, $M$ is a steady Ricci soliton. Now using Lemma \ref{le1}, the result easily follows.
\end{proof}
\section*{Acknowledgment}
The authors would like to thank the unknown referee very much for carefully reading and thoughtful criticism of the manuscript which helps us to improve the manuscript.

\end{document}